\documentclass[11pt]{article}

\usepackage{amssymb, amsmath, amsthm, graphicx}
\usepackage[left=1in,top=1in,right=1in]{geometry}

\usepackage{subfigure}

   \theoremstyle{plain}
      \newtheorem{theorem}{Theorem}[section]
      \newtheorem{lemma}[theorem]{Lemma}
            \newtheorem{observation}[theorem]{Observation}
      \newtheorem{corollary}[theorem]{Corollary}
      
\newtheorem{conjecture}[theorem]{Conjecture}
      \theoremstyle{definition}
      
      \theoremstyle{remark}

	\def\ocn{\mbox{\rm odd-cr}}

\author{Radoslav Fulek\thanks{The authors gratefully acknowledge support from the Swiss National
     Science Foundation Grant No.  200021-125287/1}
\thanks{EPFL, Lausanne.
Email: {\tt radoslav.fulek@epfl.ch}} \and Andrew Suk$^*$\thanks{Courant Institute, New York and EPFL, Lausanne. Email: {\tt suk@cims.nyu.edu}}
}

\title{On disjoint crossing families in geometric graphs}

\begin{document}

\maketitle

\begin{abstract}
A \emph{geometric graph} is a graph drawn in the plane with vertices
represented by points and edges as straight-line segments.  A geometric graph contains a \emph{$(k,l)$-crossing family}
 if there is a pair of edge subsets $E_1,E_2$ such that $|E_1| = k$ and $|E_2| = l$,
 the edges in $E_1$ are pairwise crossing, the edges in $E_2$ are pairwise crossing,
 and every edges in $E_1$ is disjoint to every edge in $E_2$.  We conjecture that for any
 fixed $k,l$, every $n$-vertex geometric graph with no $(k,l)$-crossing family has at most
$c_{k,l}n$ edges, where $c_{k,l}$ is a constant that depends only on $k$ and $l$.
 In this note, we show that every $n$-vertex geometric graph with no $(k,k)$-crossing family has at most $c_kn\log n$ edges,
 where $c_k$ is a constant that depends only on $k$, by proving a more general result which relates
 extremal function of a geometric graph $F$ with extremal function of two completely disjoint copies of $F$.
We also settle the conjecture for geometric graphs with no $(2,1)$-crossing family.  As a direct application, this implies that for any circle graph $F$ on 3 vertices, every $n$-vertex geometric graph that does not contain a matching whose intersection graph is $F$ has at most $O(n)$ edges.

\end{abstract}

\section{Introduction}

A \emph{topological graph} is a graph drawn in the plane with points
as vertices and edges as non-self-intersecting arcs connecting its vertices. The arcs are
allowed to intersect, but they may not pass through vertices except for
their endpoints.  Furthermore, the edges are not allowed to have tangencies, i.e., if two edges share an interior point, then they must properly cross at that point.
We only consider graphs without parallel edges or self-loops.
A topological graph is \emph{simple} if every pair of its edges intersect at most once.
If the edges are drawn as straight-line segments, then the graph is \emph{geometric}.  Two edges of a topological graph \emph{cross} if their interiors share a point, and are \emph{disjoint} if they do not have a point in common (including their endpoints).

It follows from Euler's Polyhedral Formula that every simple topological graph on $n$ vertices and no crossing edges has at most $3n -6$ edges.  It is also known that every simple topological graph on $n$ vertices
 with no pair of disjoint edges has at most $O(n)$ edges \cite{lov},\cite{fulek}.  Finding the maximum number of edges in a topological (and geometric) graph with a forbidden substructure has been a classic problem in extremal topological graph theory (see \cite{ack}, \cite{ag}, \cite{pach}, \cite{foxpach}, \cite{toth}, \cite{rados}, \cite{tardos}, \cite{pacht}, \cite{tothvaltr}).  Many of these problems ask for the maximum number of edges in a topological (or geometric) graph whose edge set does not contain 
a matching that defines a particular intersection graph.  Recall that the {\it intersection graph} of objects $\mathcal{C}$ in the plane is a graph with vertex set $\mathcal{C}$, and two vertices are adjacent if their corresponding objects intersect.  Much research has been devoted to understanding the clique and independence number of intersection graphs due to their applications in VLSI design \cite{vlsi}, map labeling \cite{map}, and elsewhere.

Recently, Ackerman et al. \cite{afps} defined a \emph{natural $(k,l)$-grid} to be a set of $k$ pairwise disjoint edges that all cross another set of $l$ pairwise disjoint edges.  They conjectured

\begin{conjecture}[]
\label{con:grid}
Given fixed constants $k,l \geq 1$ there exists another constant $c_{k,l}$,
such that any geometric graph on $n$ vertices with no natural $(k,l)$-grid
has at most $c_{k,l}n$ edges.
\end{conjecture}

\noindent They were able to show,

\begin{theorem}
\label{grid1}
\emph{\cite{afps}} For fixed $k$, an $n$-vertex geometric graph with no natural $(k,k)$-grid
has at most $O(n\log^2 n)$ edges.
\end{theorem}

\begin{theorem}
\label{grid2}
\emph{\cite{afps}} An $n$-vertex geometric graph with no natural $(2,1)$-grid
has at most $O(n)$ edges.
\end{theorem}

\begin{theorem}
\label{grid3}
\emph{\cite{afps}} An $n$-vertex simple topological graph with no natural $(k,k)$-grid
has at most $O(n\log^{4k-6}n)$ edges.
\end{theorem}

\noindent As a \emph{dual} version of the natural $(k,l)$-grid, we define a \emph{$(k,l)$-crossing family} to be a pair of edge subsets $E_1,E_2$ such that

\begin{enumerate}

\item $|E_1| = k$ and $|E_2| = l$,

\item the edges in $E_1$ are pairwise crossing,

\item the edges in $E_2$ are pairwise crossing,

\item every edge in $E_1$ is disjoint to every edge in $E_2$.

\end{enumerate}

\noindent We conjecture

\begin{conjecture}[]
\label{con}
Given fixed constants $k,l \geq 1$ there exists another constant $c_{k,l}$,
such that any geometric graph on $n$ vertices with no $(k,l)$-crossing family has at most $c_{k,l}n$ edges.
\end{conjecture}

\noindent It is not even known if all $n$-vertex geometric graphs with no $k$ pairwise crossing edges has $O(n)$ edges.
  The best known bound is due to Valtr \cite{valtr}, who showed that this is at most $O(n\log n)$ for every fixed $k$.
  We extend this result to $(k,k)$-crossing families by proving the following theorem.

\begin{theorem}
\label{geo}
An $n$-vertex geometric graph with no $(k,k)$-crossing family
has at most $c_kn\log n$ edges, where  $c_k$ is a constant that depends only on $k$.
\end{theorem}

Let $F$ denote a geometric graph. 
We say that a geometric graph $G$ contains $F$ as a geometric subgraph if $G$ contains a subgraph $F'$ isomorphic to $F$ such that two edges in $F'$ cross
if and only if the two corresponding edges cross in $F$.

We define $ex(F,n)$ to be the extremal function of $F$, i.e.
the maximum number of edges a geometric graph on $n$ vertices can have without
containing $F$ as a geometric subgraph.  Similarly, we define $ex_L(F,n)$  to be the extremal function of $F$,
if we restrict ourselves to the geometric graphs all of whose edges can be hit by one line.

Let $F_2$ denote a geometric graph,
which consists of two completely disjoint copies of a geometric graph $F$.
We prove Theorem \ref{geo} by a straightforward application of the following result.

 \begin{theorem}
\label{thm:disjoint}
$ex(F_2,n)= O((ex_L(F,2n)+n)\log n+ex(F,n))$
\end{theorem}

\noindent Furthermore, we settle Conjecture~\ref{con} in the first nontrivial case.

\begin{theorem}
\label{thm:21fam}
An $n$-vertex geometric graph with no $(2,1)$-crossing family
has at most $O(n)$ edges.
\end{theorem}

\noindent Note that Conjecture~\ref{con} is not true for topological graphs since Pach and T\'oth \cite{pachtoth}
 showed that the complete graph can be drawn such that every pair of edges intersect once or twice.

Recall that $F$ is a \emph{circle graph} if $F$ can be represented as the intersection graph of chords on a circle.  By combining Theorem \ref{thm:21fam} with results from \cite{ag}, \cite{afps}, and \cite{toth}, we have the following.

\begin{corollary}
For any circle graph $F$ on $3$ vertices, every $n$-vertex geometric graph that
 does not contain a matching whose intersection graph is $F$ contains at most $O(n)$ edges.
\end{corollary}

  \begin{figure}
  \centering
\subfigure[3 pairwise crossing.]{\label{3cross}\includegraphics[width=0.2\textwidth]{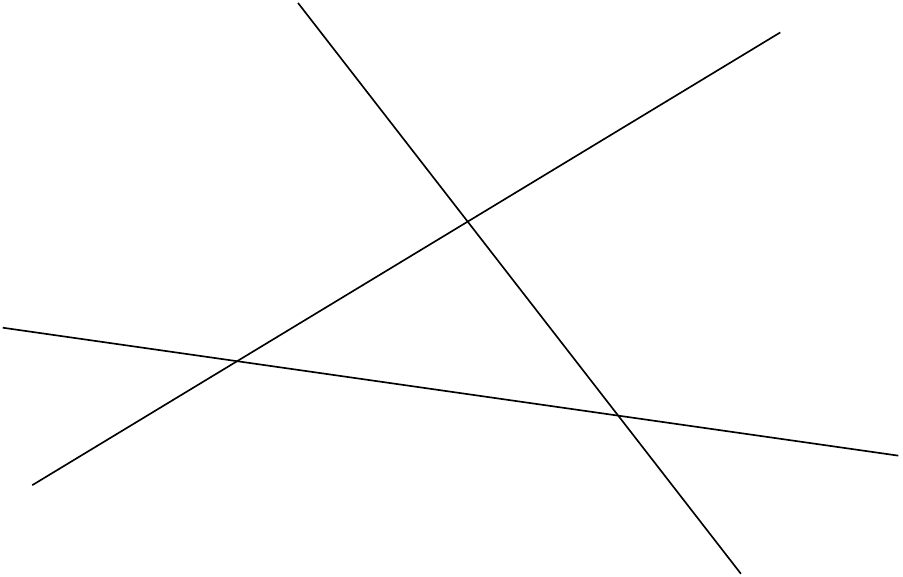}}\hspace{1cm}
\subfigure[3 pairwise disjoint.]{\label{3disjoint}\includegraphics[width=0.15\textwidth]{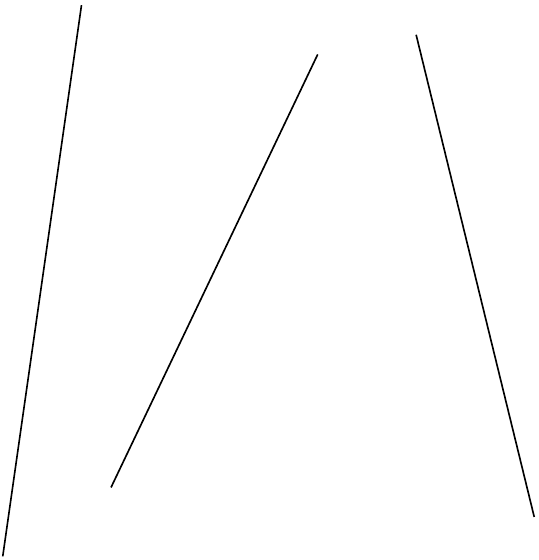}}\hspace{1cm}
\subfigure[(2,1)-grid.]{\label{21grid}\includegraphics[width=0.2\textwidth]{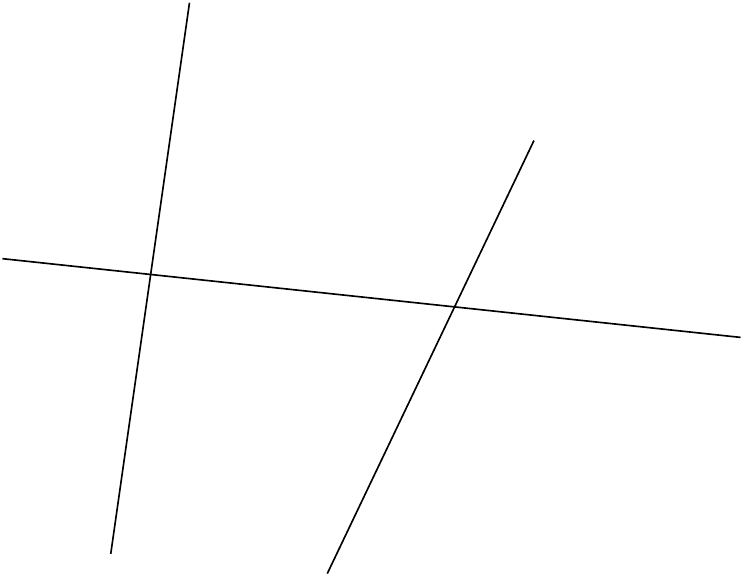}}\hspace{1cm}
\subfigure[(2,1)-crossing family.]{\label{21fam}\includegraphics[width=0.15\textwidth]{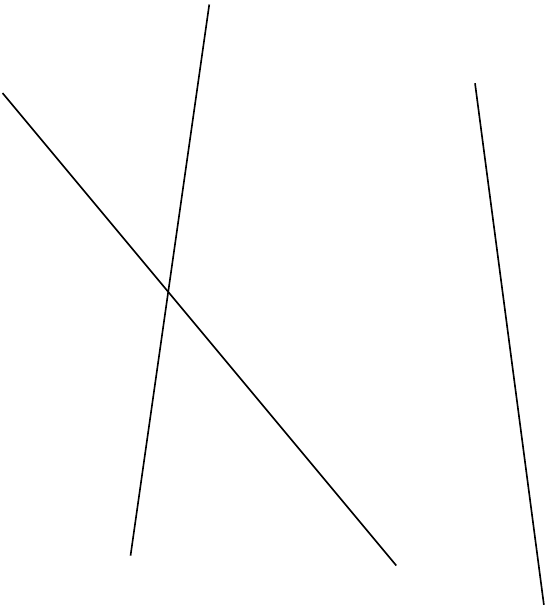}}
\caption{Triples of segments corresponding to all circle graphs on three vertices.}
  \label{fig:3match}
\end{figure}

\noindent  See Figure \ref{fig:3match}.
We also conjecture the following.

\begin{conjecture}
\label{circle}
For any circle graph $F$ on $k$ vertices, there exists a constant
 $c_k$ such that every $n$-vertex geometric graph that does not contain a
 matching whose intersection graph is $F$, contains at most $c_kn$ edges.
\end{conjecture}

\noindent As pointed out by Klazar and Marcus \cite{klazar},
 it is not hard to modify the proof of the Marcus-Tardos Theorem \cite{tardos} to show that
Conjecture \ref{circle} is true when the vertices are in convex position.

For simple topological graphs, we have the following

\begin{theorem}
\label{top}
An $n$-vertex simple topological graph with no $(k,1)$-crossing family
has at most $n(\log n)^{O(\log k)}$ edges.
\end{theorem}

The paper is organized as follows. Section 2 is devoted to the proof of Theorem~\ref{thm:disjoint}.
In Section 3 we establish Theorem~\ref{thm:21fam}. Finally, the result of Theorem~\ref{top} about topological graphs
is proved in Section 4.

\section{Relating extremal functions}

First, we prove a variant of Theorem~\ref{thm:disjoint} when all of the edges in our geometric graph
can be hit by a line.
As in the introduction let $F_2$ denote a geometric graph, which consists of two completely disjoint copies of a geometric graph $F$.
We will now show that the extremal function $ex_L(F_2,n)$ is not far from $ex_L(F,n)$.

\begin{theorem}
\label{thm:disjointL}
$ex_L(F_2,n)\leq  O((n+ex_L(F,2n))\log n).$
\end{theorem}

\begin{proof}

Let $G$ denote a geometric graph on $n$ vertices that does not contain
 $F_2$ as a geometric subgraph, and all the edges of $G$ can be hit by a line.
By a standard perturbation argument we can assume that the vertices of $G$ are in general position.  As in \cite{Dey}, a \emph{halving edge} $uv$ is a pair of the vertices in $G$
such that the number of vertices on each side of the line through $u$ and $v$ is the same.

\begin{lemma}
\label{lemma:halving}
There exists a directed line $\vec{l}$ such that the number of edges in $G$ that lies completely to the left or right of $\vec{l}$ is at most $2ex_L(F,n/2)+5n$.
\end{lemma}

\begin{proof}
If $n$ is odd we can discard one vertex of $G$, thereby loosing at most $n$ edges.
Therefore we can assume $n$ is even, and it suffices to show that there exists a directed line $\vec{l}$ such that the number of edges in $G$ that lies completely to the left or right of $\vec{l}$ is at most $2ex_L(F,n/2)+4n$.

 Let $uv$ be a halving edge, and let $\vec{l}$ denote the directed line containing vertices $u$ and $v$ with direction from $u$ to $v$.  Let $e(\vec{l},L)$ and $e(\vec{l},R)$ denote the number of edges on the left and right side of $\vec{l}$ respectively.  Without loss of generality, we can assume that $e(\vec{l},L) \leq e(\vec{l},R)$.   We will rotate $\vec{l}$ such that it remains a halving line at the end of each step, until it reaches a position where the number of edges on both sides of $\vec{l}$ is roughly the same.

 We start by rotating $\vec{l}$ counterclockwise around $u$ until it meets the next vertex $w$ of $G$.  If initially $w$ lies to the right of $\vec{l}$, then in the next step we will rotate $\vec{l}$ around $u$ (again).  See Figure  \ref{fig:halvingline2}.  Otherwise if $w$ was on the left side of $\vec{l}$, then in the next step we will rotate $\vec{l}$ around vertex $w$.  See Figure \ref{fig:halvingline}.  Clearly after each step in the rotation, there are exactly $n/2$ vertices on each side of $\vec{l}$.

\begin{figure}[h]
\centering
\subfigure[$w$ lies to the right of $\vec{l}$.]{\includegraphics[scale=0.55]{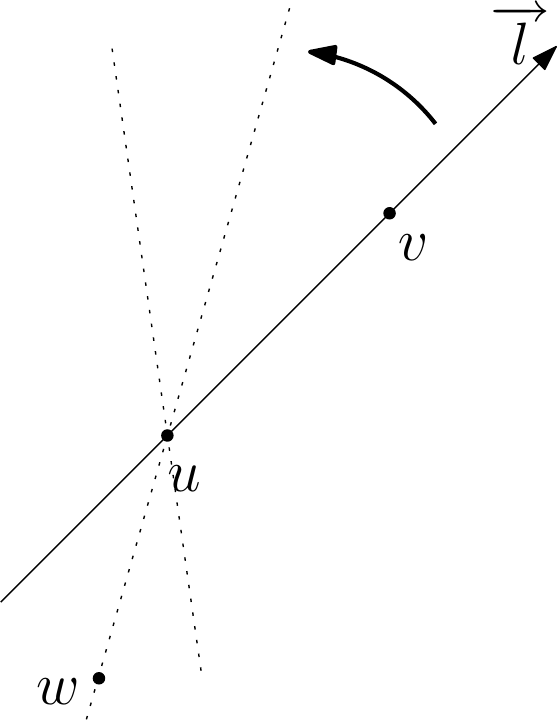}
	\label{fig:halvingline2}
	}\hspace{2cm}
\subfigure[$w$ lies to the left of $\vec{l}$.]{\includegraphics[scale=0.55]{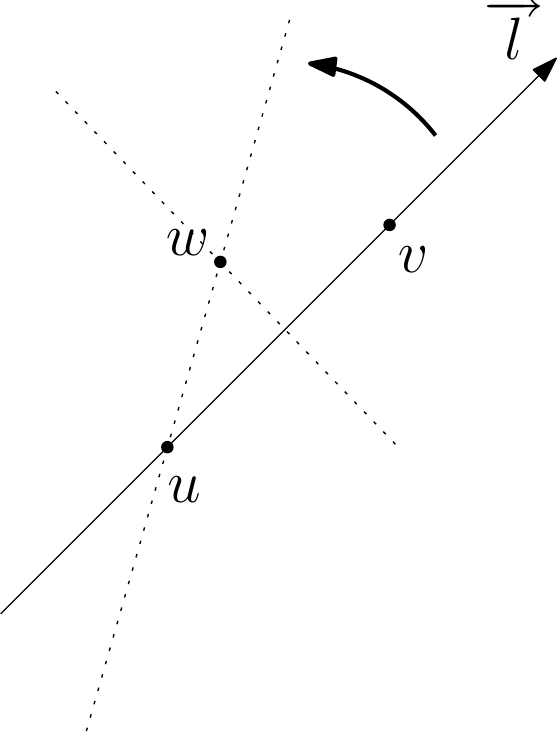}
	\label{fig:halvingline}
	}
\caption{Halving the vertices of $G$}
\end{figure}

After several rotations, $\vec{l}$ will eventually contain points $u$ and $v$ again, with direction from $v$ to $u$.  At this point we have $e(\vec{l},L) \geq e(\vec{l},R)$.  Since the number of edges on the right side (and left side) changes by at most $n$ after each step in the rotation, at some point in the rotation we must have

$$|e(\vec{l},L) - e(\vec{l},R)| \leq 2n.$$

Since $G$ does not contain a $F_2$ as a geometric subgraph, this implies that

$$e(\vec{l},L) \leq ex_L(F,n/2) + 2n$$

\noindent and

$$e(\vec{l},R) \leq ex_L(F,n/2) + 2n.$$

\noindent Therefore for any $n$, there exists a directed line $\vec{l}$ such that the number of edges in $G$ that lies completely to the left or right of $\vec{l}$ is at most $2ex_L(F,n/2)+5n$.

\end{proof}

\begin{figure}
\label{fig:hamsandwich}
\centering
\includegraphics[scale=0.6]{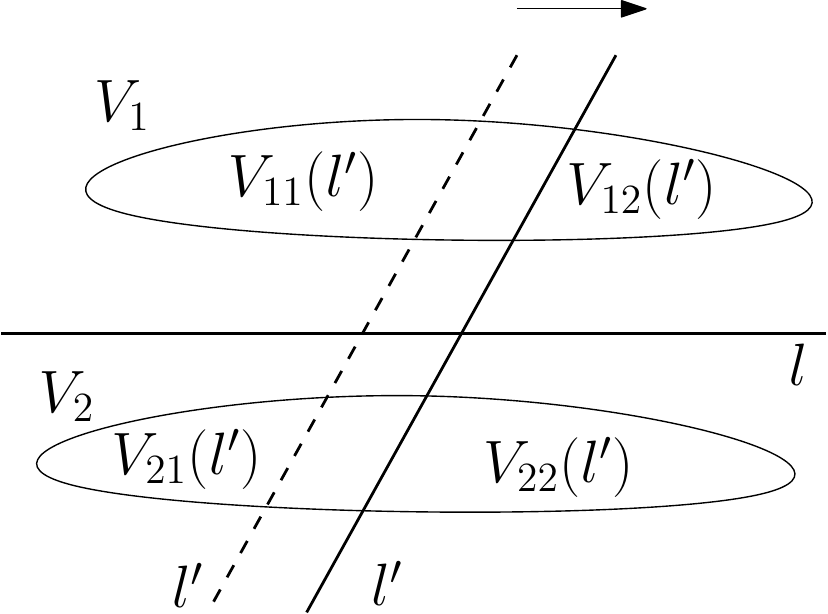}
\caption{The final partition of the vertex set of $G$}
\end{figure}

By Lemma \ref{lemma:halving} we obtain a line $l$, which partition the vertices of $G$ into two  equal
(or almost equal if $n$ is odd) sets $V_1$ and $V_2$.
Let $E'$ denote the set of edges between $V_1$ and $V_2$.
By the Ham-Sandwich Cut Theorem \cite{mat},
there exists a line $l'$ that simultaneously bisects $V_1$ and $V_2$.
Let $V_{11}(l')$ and $V_{12}(l')$ denote the resulting parts of $V_1$, and let $V_{21}(l')$ and $V_{22}(l')$
denote the resulting parts of $V_2$.

Observe that we can translate $l'$ along $l$ into a position where
 the number of edges in $E'$ that lie completely to the left and completely to the right of $l'$ is roughly the same.  In particular, we can translate $l'$ along $l$ such that the number of edges in $E'$ that lies completely to its left or right side is at most $ex_L(F,n)+ex_L(F,n/2+1)+n$ (see Figure
\ref{fig:hamsandwich}).  Indeed, assume that the number of edges in $E'$ between, say,
$V_{12}(l')$ and $V_{22}(l')$ is more than $ex_L(F,n/2+1)$.  As we translate $l'$ to the right, the number of edges that lie completely to the right of $l'$ changes by at most $n$ as $l'$ crosses a single vertex in $G$.  Therefore we can translate $l'$ into the leftmost position where the number of edges in $E'$ between
$V_{12}(l')$ and $V_{22}(l')$ drops below  $ex_L(F,n/2+1)+n+1$.  Moreover, at this position the number of
edges in $E'$ between $V_{11}(l')$ and $V_{21}(l')$ still cannot be more than
$ex_L(F,n)$ since $G$ does not contain $F_2$ as a geometric subgraph.

Thus, all but at most $3ex_L(F,n/2+1)+ex_L(F,n)+6n$ edges of $G$ are the edges between $V_{11}(l')$ and $V_{22}(l')$,
 and between $V_{12}(l')$ and $V_{21}(l')$.
Notice that there exists $k$, $-1/4\leq k\leq 1/4$,
 such that $|V_{11}(l')|+|V_{22}(l')|=n(1/2+k)$, and $|V_{12}(l')|+|V_{21}(l')|=n(1/2-k)$.
Finally, we are in the position to state the recurrence, whose closed form gives the statement of the theorem.

$$ex_L(F_2,n)\leq ex_L(F_2,n(1/2+k))+ex_L(F_2,n(1/2-k)) + 3ex_L(F,n/2+1)+ex_L(F,n)+6n.$$

\noindent By a routine calculation which is indicated below, we have
\begin{eqnarray}
\nonumber
ex_L(F_2,n)&\leq &  \log_{\frac{4}{3}}\left(n\left(\frac{1}{2}+k\right)\right)\left(6n\left(\frac{1}{2}+k\right)+
 4ex_L\left(F_2,2n\left(\frac{1}{2}+k\right)\right)\right) + \\ \nonumber
 & & \log_{\frac{4}{3}}\left(n\left(\frac{1}{2}-k\right)\right)\left(6n\left(\frac{1}{2}-k\right)+
 4ex_L\left(F_2,2n\left(\frac{1}{2}-k\right)\right)\right) + \\ \nonumber
 &&  4ex_L(F,n)+ 6n \\ \nonumber
 &\leq & \log_{\frac{4}{3}}n(4ex_L(F,2n)+ 6n)
\end{eqnarray}

\end{proof}

Finally, we show how Theorem~\ref{thm:disjointL} implies Theorem~\ref{thm:disjoint}.

\begin{proof}[Proof of Theorem \ref{thm:disjoint}.]
Let $G=(V,E)$ denote the geometric graph not containing $F_2$ as a subgraph.  Similarly, as in the proof of Lemma \ref{lemma:halving} we can find a halving line $l$ that hits
all but $2ex(F,n/2)+5n$ edges of $G$. Now, the claim follows by using Theorem \ref{thm:disjointL}.
\end{proof}

Theorem~\ref{geo} follows easily by using Theorem \ref{thm:disjoint}
 with a result from \cite{valtr}, which states that every $n$-vertex geometric graph whose
 edges can be all hit by a line and does not contain $k$ pairwise crossing edges has at most
$O(n)$ edges and at most $O(n\log n)$ edges if we do not require a single line to hit all the edges.

\section{Geometric graphs with no (2,1)-crossing family}

In this section we will prove Theorem \ref{thm:21fam}.  Our main tool is the following theorem by T\'oth and Valtr

\begin{theorem}
\label{thm:matching}
\emph{\cite{tothvaltr}} Let $G =(V,E)$ be an $n$-vertex geometric graph.
 If $G$ does not contain a matching consisting of $5$ pairwise disjoint edges, then $|E(G)| \leq 64n + 64$.
 \end{theorem}

\noindent Theorem \ref{thm:21fam} immediately follows from the following theorem.

\medskip

\begin{theorem}
\label{21proof}
Every $n$-vertex geometric graph with no $(2,1)$-crossing family has at most $64n  + 64$ edges.

\end{theorem}

\begin{proof}
For sake of contradiction, let $G =(V,E)$ be a vertex-minimal counter example, i.e. $G$ is a geometric graph on $n$ vertices which has more than $64n  + 64$
edges and
$G$ does not contain a (2,1)-crossing family.  Hence every vertex in $G$ has degree at least 65.  Let $M$ denote the maximum matching in $G$ consisting of pairwise disjoint edges and let $V_M$ denote the vertices in $M$.
Since $|E(G)| > 64n + 64$, Theorem~\ref{thm:matching} implies that $|M|\geq 5$.
We say that two edges \emph{intersect} if they cross or share an endpoint.
The following simple observation is crucial in the subsequent analysis.\\

\begin{tabular}{rl}
(*) & An edge $e\in E$ that crosses an edge of $M$ must intersect every edge in $M$.
\end{tabular} \\

\noindent Indeed, otherwise we would obtain a (2,1)-crossing family.
  We call an endpoint $v$ of an edge in $M$ \emph{good} if every ray starting at $v$ misses at least one edge in $M$.  See Figure \ref{good}.

\begin{lemma}
\label{lemma:lama}
For $|M| \geq 4$, at least $|M|-2$ of the endpoints in $M$ are good.
\end{lemma}

\begin{proof}
We proceed by induction on $|M|$.  Assume $|M| = 4$.  If every triple in $M$ has a good vertex, then clearly we have at least two good vertices.  Otherwise the only matching consisting of three pairwise disjoint
edges with no good vertices is the one in Figure \ref{mustbe}.
By a simple case analysis, adding a disjoint edge to this matching creates two good vertices (see Appendix \ref{app:4edges}).
For the inductive step when $|M|>4$, we choose an arbitrary 4-tuple of edges in the matching.
By the above discussion, the 4-tuple has at least one good endpoint.
By removing the edge incident to this good vertex, the statement follows by the induction hypothesis.

\end{proof}

\begin{figure}[h]
  \centering
  \subfigure[]{
    \includegraphics[width=.2\textwidth]{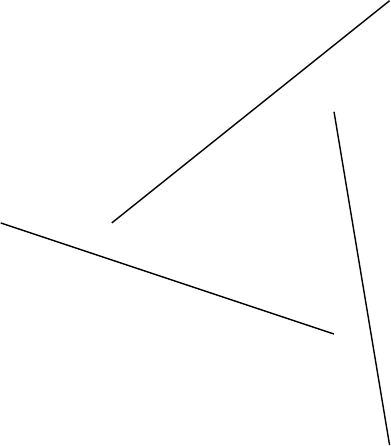}      
      \label{mustbe}}
      \hspace{100px}
   \subfigure[]{   
   \includegraphics[width=.2\textwidth]{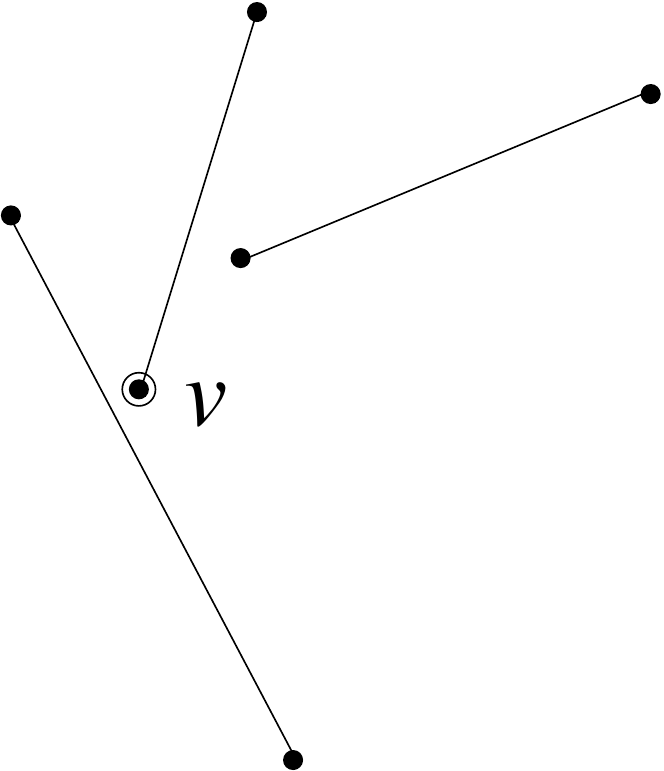}
      \label{good}}
   \caption{(a) Special case in Lemma \ref{lemma:lama}, (b) Matching $M$ of size 3 with one good vertex $v$}
\end{figure}

 Thus by (*), a good endpoint cannot be incident to an edge that crosses any of the edges in $M$.
 Let

 \begin{enumerate}

 \item $V_g\subseteq V_M$ denote the set of good endpoints in $M$.

 \item $V_1\subset V\setminus V_M$ be the subset of the vertices such that for $v \in V_1$, every edge incident to $v$ does not cross any of the edges in $M$,

 \item and $V_2 = V\setminus (V_M\cup V_1)$.  Hence for $v\in V_2$, there exists an edge incident to $v$ that intersects every edge in $M$.

\end{enumerate}

\begin{figure}[h]
  
\end{figure}

\noindent See Figure \ref{setup}.  By Lemma \ref{lemma:lama}, $|V_g| \geq |M|-2 = |V_M|/2 - 2$.   By maximality of $M$, there are no edges between $V_1$ and $V_2$ and $V_1$ is an independent set.  Now notice the following observation.

\begin{observation}

There exists a good vertex in $V_g$ that has at least three neighbors in $V_2$.

\end{observation}

\medskip

\noindent \emph{Proof.}  For sake of contradiction, suppose that each vertex in $V_g$ has at most two neighbors in $V_2$.  Then let $G'=(V', E')$ denote a subgraph of $G$ such that $V' = V_M\cup V_1$ and $E'$
 consists of the edges that do not cross any of the edges of $M$ and whose endpoints are in $V_1\cup V_M$. Since  $|M|\geq 5$, $G'$ must be a planar graph since otherwise we would have a $(2,1)$-crossing family.  Therefore $E'\leq 3(|V_1|+|V_M|)$.

On the other hand, by minimality of $G$, each vertex in $V_1$ has degree at least 65 in $G'$, and each vertex in $V_g$ has degree at least 63 in $G'$.  Therefore by applying Lemma \ref{lemma:lama}, we have

$$\frac{1}{2}\left( 65|V_1| + 63\left(\frac{|V_M|}{2} - 2\right)\right) \leq |E'| \leq 3|V_1| + 3|V_M|.$$

 \noindent This implies

$$59|V_1| + 25|V_M| \leq 126$$

\noindent which is a contradiction since $|V_M| \geq 10$ ($|M| \geq 5$).

$\hfill\square$

Let $v \in V_g$ be a good vertex such that $v$ has at least 3 neighbors in $V_2$.  Let $e=vv_1$, $f=vv_2$, $g=vv_3$, and $m$ be edges in $G$ such that $v_1,v_2,v_3\in V_2$, $m \in M$, and $v$ is a good vertex incident to $m$.  Furthermore,
 we will assume that $g,e,m,f$ appear in clockwise order around $v$. By (*) there is an edge $e'$ incident to $v_1$ having
non-empty intersection with every edge in $M$. Similarly, we can find such an edge $f'$ incident to $v_2$ (possibly $f'=e'$).
The edges $e'$ and $f'$ must have non-empty intersection with $f$ and $e$, respectively (see Figure \ref{fig:21final}).  Otherwise
we would obtain a (2,1)-crossing family in $G$ consisting of $e,f'$ and an edge from $M$, or $e',f$ and an edge from $M$.

However, a third edge $g$ cannot have a non-empty intersection with
both $e'$ and $f'$. Hence, we obtain a (2,1)-crossing family in $G$ consisting of $g,e'$
and an edge from $M$, or $g,f'$ and an edge from $M$. Thus, there is no minimal counter example and that concludes the proof.
\end{proof}

\begin{figure}[h]
\centering
\subfigure[]{
       \includegraphics[width=.3\textwidth]{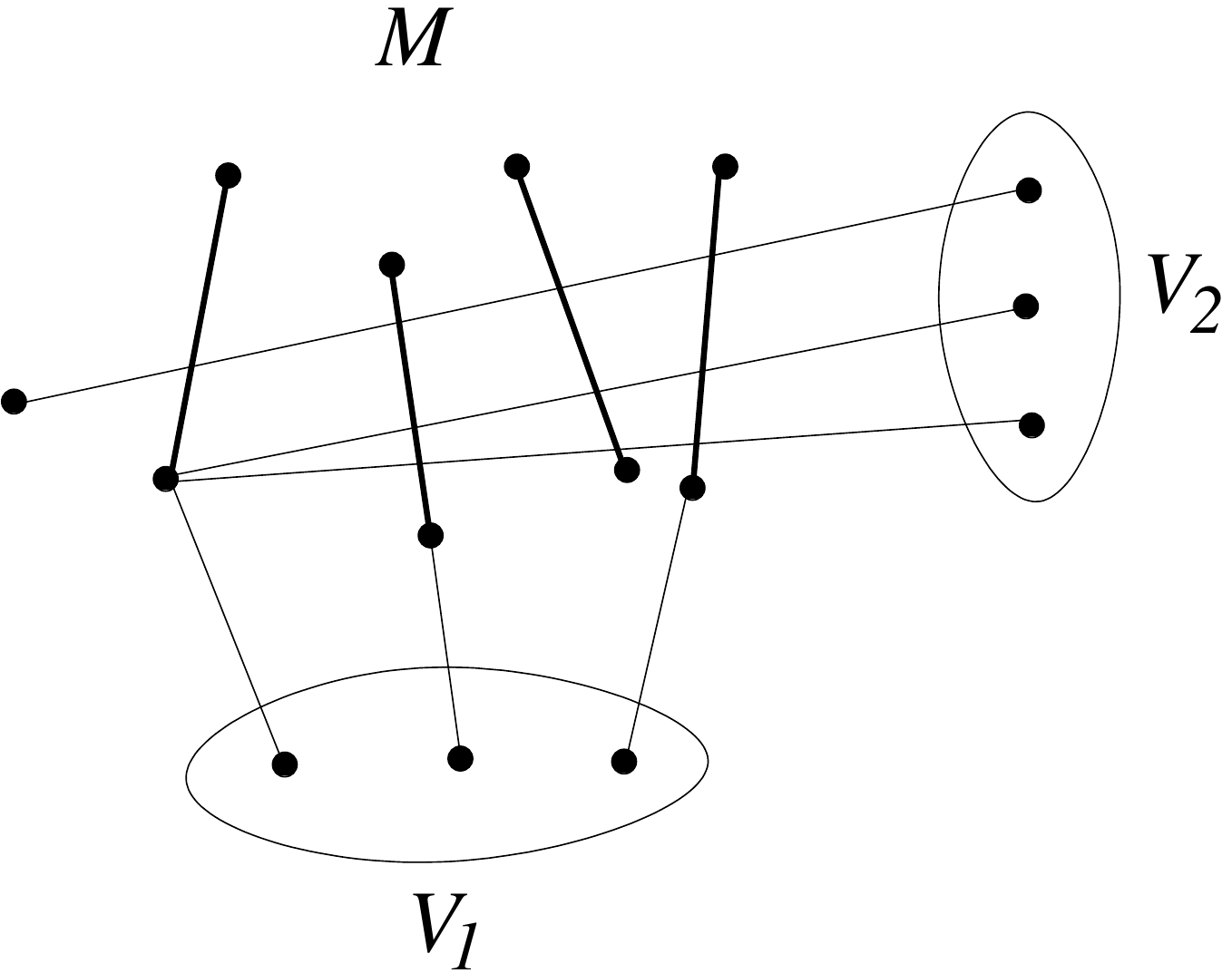}
		\label{setup}
}
\subfigure[]{
      \includegraphics[width=.3\textwidth]{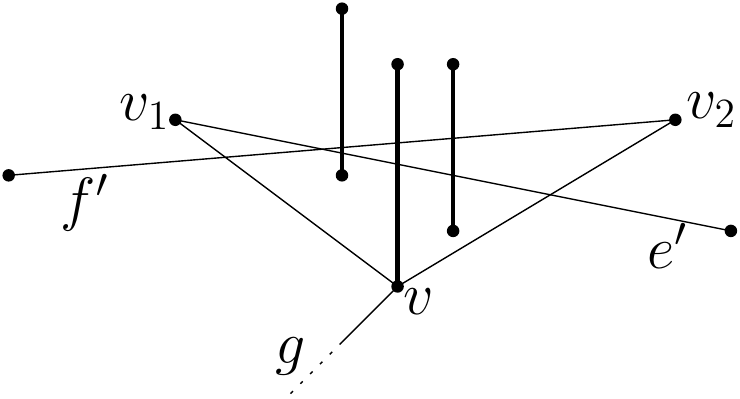}
 	  \label{fig:21final}
}
\caption{(a) $M$, $V_1$, and $V_2$, (b) Situation around the vertex $v$}
\end{figure}

\noindent We note that by a more tedious case analysis, one could improve the upper bound in Theorem \ref{21proof} to $15n$.

\section{Simple topological graphs with no $(k,1)$-crossing family}

In this section, we will prove Theorem \ref{top} which will require the following two lemmas.  The first one is due to Fox and Pach.

\begin{lemma}[ ]\label{foxpach}
\emph{\cite{foxpach}} Every $n$-vertex simple topological graph with no $k$ pairwise crossing edges has at most $n(\log n)^{c_1\log k}$ edges, where $c_1$ is an absolute constant.
\end{lemma}

\noindent

As defined in \cite{crossing}, the {\it odd-crossing number} $\ocn(G)$ of a graph $G$ is the minimum possible number of unordered pairs of edges that crosses an odd number of times over all drawings of $G$. The {\it bisection width} of a graph $G$, denoted by $b(G)$, is the smallest nonnegative integer such that there is a partition of
 the vertex set $V=V_1 \, \dot{\cup} \, V_2$ with $\frac{1}{3}\cdot |V|\leq V_i\leq \frac{2}{3}\cdot |V|$ for $i=1,2$, and  $|E(V_1,V_2)|= b(G)$. The second lemma required is due to Pach and T\'oth, which relates the odd-crossing number of a graph to its bisection width.

\begin{lemma}[]\label{bisect}
\emph{\cite{pachtoth}} There is an absolute constant $c_2$ such that if $G$ is a graph with
 $n$ vertices of degrees $d_1,\ldots,d_n$, then
 $$b(G)\leq c_2\log n \sqrt{\ocn(G)+\sum_{i=1}^n d_i^2} .$$
 \end{lemma}

\medskip

 \noindent Since all graphs have a bipartite subgraph with at least half of its edges, Theorem~\ref{geo} immediately follows from the following Theorem.

 \begin{theorem}
 \label{proof}
 Every $n$ vertex simple topological bipartite graph with no $(k,1)$-crossing family has at most $c_3n\log^{c_4\log k}n$ edges, where $c_3,c_4$ are absolute constants.

 \end{theorem}

 \noindent \emph{Proof.}  We proceed by induction on $n$.  The base case is trivial.  For the inductive step, the proof falls into two cases.

 \medskip

 \noindent \emph{Case 1.}  Suppose there are at least $|E(G)|^2/((2c_2)^2\log^6 n)$ disjoint pair of edges in $G$.  Then by defining $D(e)$ to be the set of edges disjoint from edge $e$, we have

  $$\frac{2|E(G)|}{(2c_2)^2\log^6 n}\leq \frac{\sum\limits_{e \in E(G)} |D(e)|}{|E(G)|}$$

\noindent  Hence there exists an edge that is disjoint to at least $2|E(G)|/((2c_2)^2\log^6n)$ other edges.  By Lemma~\ref{foxpach} we have

    $$\frac{2|E(G)|}{(2c_2)^2 \log^6 n} \leq n(\log n)^{c_1\log k},$$

    \noindent which implies $|E(G)| \leq c_3n\log^{c_4\log k}n$ for sufficiently large constants $c_3,c_4$.

 \medskip

\noindent \emph{Case 2.}  Suppose there are at most $|E(G)|^2/((2c_2)^2\log^6n)$ disjoint pair of edges in $G$.  Since $G$ is bipartite, let $V_a$ and $V_b$ be its vertex class.  By applying a suitable homeomorphism to the plane, we can redraw $G$ such that

 \begin{enumerate}

 \item the vertices in $V_a$ are above the line $y = 1$, the vertices in $V_b$ are below the line $y = 0$,

 \item edges in the strip $0 \leq y \leq 1$ are vertical segments,

   \item we have not created nor removed any crossings.
 \end{enumerate}

\begin{figure}
\centering
\includegraphics[width=.7\textwidth]{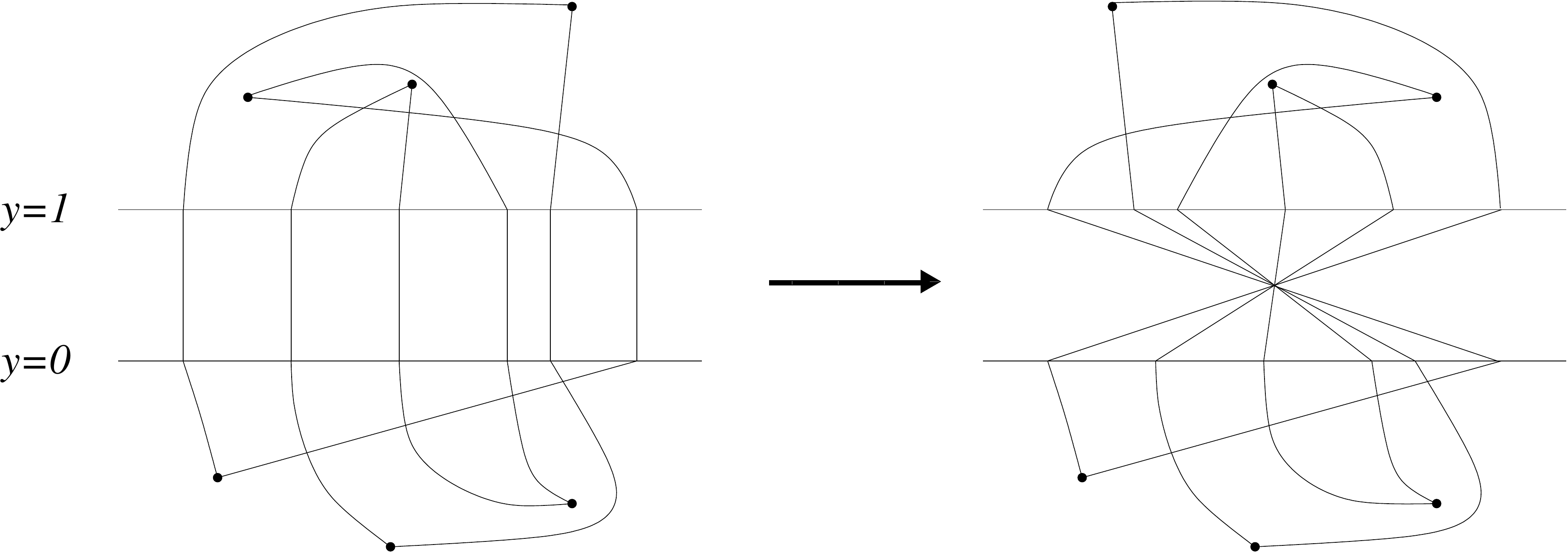}
\label{fig:redrawing}
\caption{Redrawing procedure}
\end{figure}

   \noindent  Now we reflect the part of $G$ that lies above the $y = 1$ line about the $y$-axis. 
 Then erase the edges in the strip $0\leq y \leq 1$ and replace them by straight line segments that 
reconnects the corresponding pairs on the line $y = 0$ and $y = 1$. See Figure~\ref{fig:redrawing}, and note that our graph is no longer simple.  Since there are at most $\sum\limits_{v\in V(G)} d^2(v) \leq 2|E(G)|n$ pair of edges that share a vertex in $G$, this implies

$$\ocn(G) \leq \frac{|E(G)|^2}{(2c_2)^2\log^6n} + 2|E(G)|n.$$

\noindent   By Lemma~\ref{bisect}, there is a partition of the vertex set $V=V_1 \, \dot{\cup} \, V_2$ with $\frac{1}{3}\cdot
 |V|\leq V_i\leq \frac{2}{3}\cdot |V|$ for $i=1,2$ and

 $$b(G) \leq c_2\log n \sqrt{\frac{|E(G)|^2}{(2c_2)^2\log^6 n} +  4n|E(G)|  }.$$

 \noindent  If

 $$\frac{|E(G)|^2}{(2c_2)^2\log^6 n} \leq  4n|E(G)|$$

 \noindent then we have $|E(G)| \leq c_3n\log^{c_4\log k}n$ and we are done.  Therefore we can assume

   $$ b(G) \leq c_2\log n \sqrt{\frac{2|E(G)|^2}{(2c_2)^2\log^6 n} } \leq \frac{ |E(G)|}{\log^2 n}.$$

\noindent Let $|V_1| = n_1$ and $|V_2| = n_2$.  By the induction hypothesis we have

 $$
 \begin{array}{ccl}
 |E(G)| & \leq & b(G)+  c_3n_1\log^{c_4\log k}n_1 + c_3n_2\log^{c_4\log k}n_2\\\\

    & \leq &  \frac{|E(G)|}{ \log^2  n}  + c_3n\log^{c_4\log k}(2n/3) \\\\
    & \leq &   \frac{|E(G)|}{ \log^2  n}+  c_3n(\log n  - \log(3/2))^{c_4\log k},  \\\\

 \end{array}$$

\noindent which implies

$$|E(G)| \leq c_3n\log^{c_4\log k}n \frac{(1 - \log(3/2)/\log n)^{c_4\log k}  }{1 - 1/\log^2n}  \leq c_3n\log^{c_4\log k}n.$$

 $\hfill\square$

\noindent For small values of $k$, one can obtain better bounds by replacing Lemma~\ref{foxpach}
 with a Theorem of Pach et. al. \cite{rados} and Ackerman \cite{ack} to obtain

\medskip

 \begin{theorem}
 \label{smallk}
For $k > 4$, every $n$ vertex simple topological graph with no $(k,1)$-crossing family has at most $O\left(n\log^{2k+2}n\right)$ edges.  For $k = 2,3,4$, every $n$ vertex simple topological graph with no $(k,1)$-crossing family has at most $O\left(n\log^{6}n\right)$ edges.

 \end{theorem}

$\hfill\square$

\appendix{{\bf Appendix}}

\section{Four disjoint edges}
\label{app:4edges}

\begin{figure}[h]
\centering

 \subfigure[]{
 \label{triangle1}
		\includegraphics[scale=0.5]{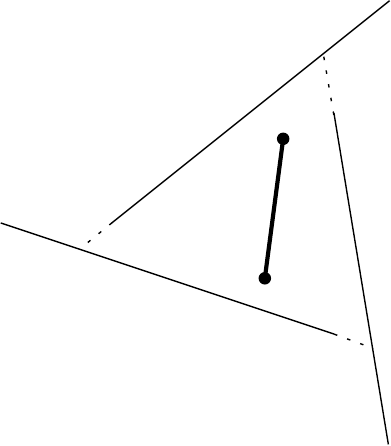}
	 \hspace{5mm}
	}
	\subfigure[]{
		\label{fig:triangle2}
		\includegraphics[scale=0.5]{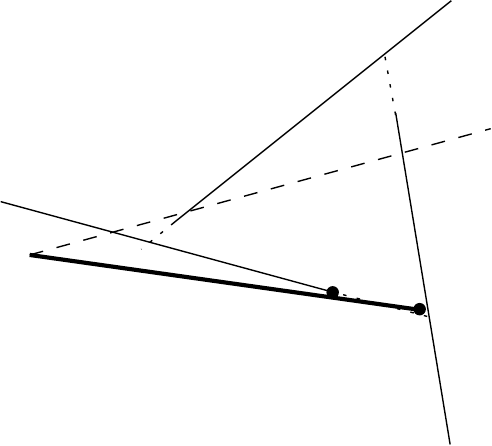}}
		\subfigure[]{
		\label{fig:triangle3}
		\includegraphics[scale=0.5]{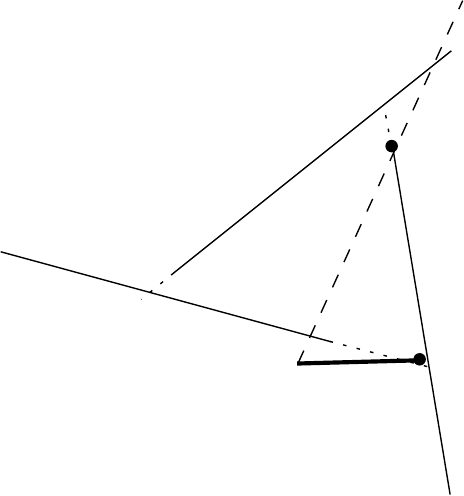}}
		\subfigure[]{
		\label{fig:triangle4}
		\includegraphics[scale=0.5]{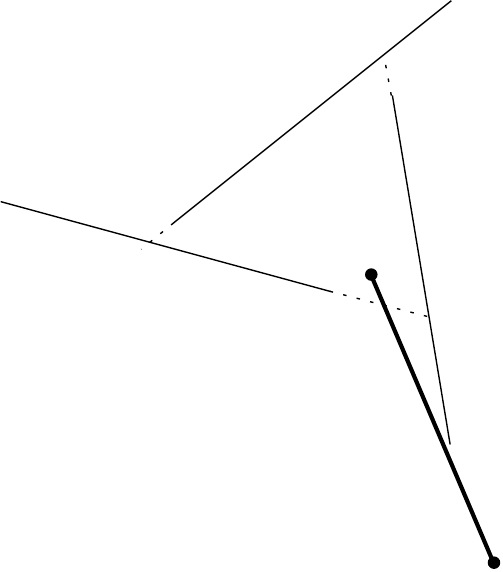}}
		\subfigure[]{
		\label{fig:triangle5}
		\includegraphics[scale=0.5]{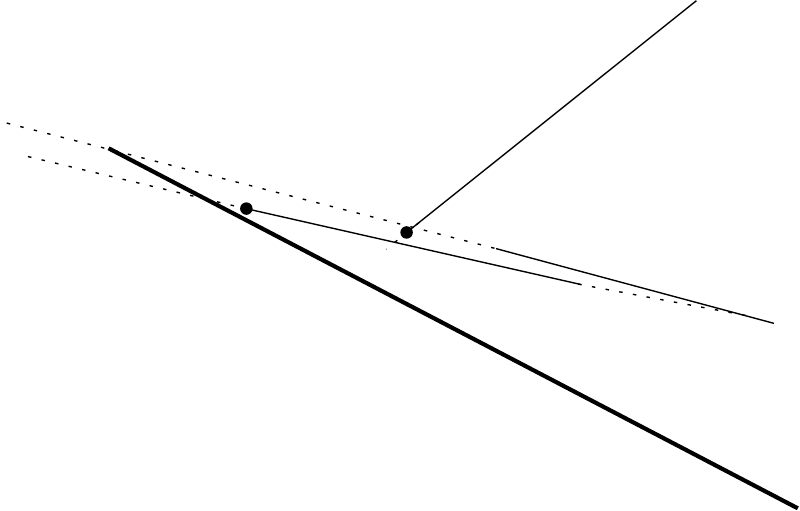}}
		\subfigure[]{
		\label{fig:triangle6}
		\includegraphics[scale=0.5]{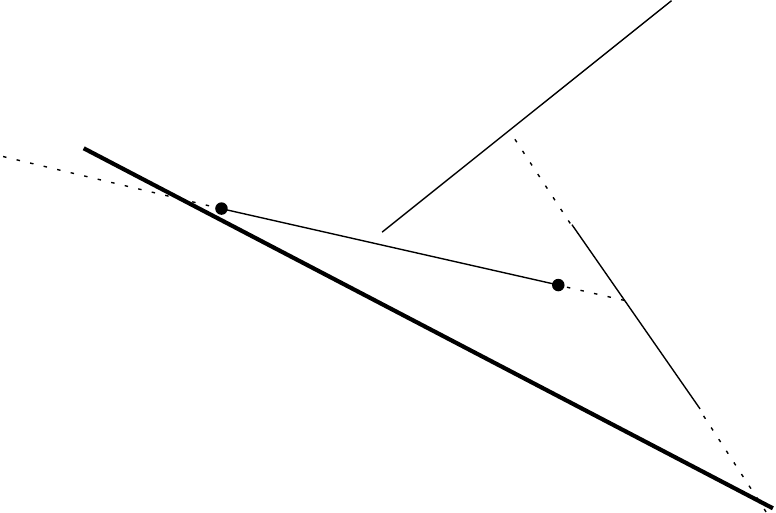}}
		\subfigure[]{
		\label{fig:triangle7}
		\includegraphics[scale=0.5]{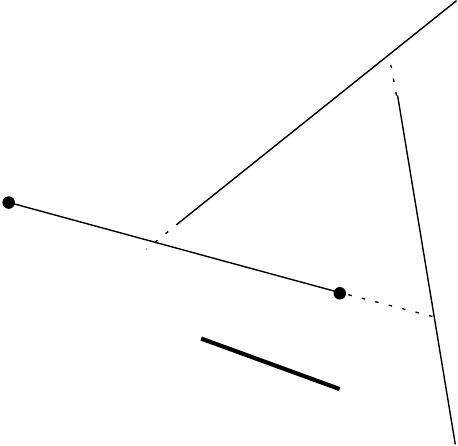}}
	\caption{Possible configuration of 4 pairwise disjoint edges, the additional edge is bold, and good vertices are marked by small discs}
\end{figure}

Suppose that we have three pairwise disjoint edges in the plane, whose combinatorial configuration
is that of the configuration in Figure \ref{mustbe}. Let $T$ denote the triangle we get
by prolonging the edges until they hit another edge.
In what follows
we show that by adding additional edge to this configuration, so that all the edges remain pairwise disjoint,
we obtain at least two good endpoints (as defined in the proof of Theorem \ref{thm:21fam}).

There are three cases to check:
\begin{enumerate}
\item
The additional edge is completely inside of the triangle $T$ (see Figure \ref{triangle1}).
The two good points are the endpoints of the additional edge.
\item
The additional  edge $e$ has one endpoint in the inside the triangle $T$ and the one outside of it.
Clearly, the endpoint of $e$ inside of $T$ is a good endpoint. If the other endpoint of $e$ is not good,
we can easily see (see Figure \ref{fig:triangle2}, \ref{fig:triangle3} and \ref{fig:triangle4}), that one of the remaining endpoints is good.
\item
The additional  edge is completely outside of the triangle $T$. There are three cases to be checked
(see Figure \ref{fig:triangle5}, \ref{fig:triangle6} and \ref{fig:triangle7}), according to where the lines through our segments meet.

\end{enumerate}

\end{document}